\documentclass[11pt,thmsa]{article}
\usepackage{amsmath, latexsym, amsfonts, amssymb, amsthm, amscd}
\usepackage{xcolor}

\newtheorem{theorem}{Theorem}

\newtheorem{lemma}{Lemma}

\newcommand{\p}{\Bbb{P}}

\newcommand{\e}{\Bbb{E}}

\newcommand{\R}{\mathbb{R}}

\newcommand{\ud}{\mathrm{d}}

\title{
\textbf{L\'evy insurance risk processes with parisian type severity of debt.}}

\author{\textbf{ J.C. Pardo\footnote{Centro de Investigaci\'on en Matem\'aticas A.C. Calle Jalisco s/n. C.P. 36240, {\sc Guanajuato, Mexico.}  Email: jcpardo@cimat.mx.}\,\,, J.L. P\'erez\footnote{Department of Probability and Statistics, IIMAS, UNAM. , C.P. 04510 {\sc Mexico, D.F., Mexico.} Email: garmendia@sigma.iimas.unam.mx}\,\,, V.M. Rivero\footnote{Centro de Investigaci\'on en Matem\'aticas A.C. Calle Jalisco s/n. C.P. 36240, {\sc Guanajuato, Mexico.}  Email: rivero@cimat.mx.}\, .}}
\date{\footnotesize This version: \today}

\begin{document}

\maketitle

\begin{abstract}
\bigskip

In this article, we introduce  a new definition of bankruptcy  for a spectrally negative L\'evy insurance risk process. More precisely, we study the  Gerber-Shiu distribution for a ruin model where at each time the surplus  goes negative, an independent negative random level is considered. If  a  negative excursion of the surplus exceeds such random level  then the insurance company goes out of business. Our methodology uses excursion theory  and relies on the description of  the excursion measure away from 0   which was recently obtained by the authors  in \cite{cem}.  Our results are given in terms of the so-called scale functions.


\bigskip

\end{abstract}

\vspace{0.5cm}

\section{Introduction and main results}

In the traditional ruin theory literature, ruin occurs immediately once the surplus of a company is negative. However, in practice if ruin occurs the company  usually continues its business and survives with a negative surplus in the hope of a quick recovery. With this idea in mind, many authors have been interested in defining new  risk models that allow the surplus of a company to stay negative without declaring ruin immediately. For instance,  Dassios and Embrecht \cite{DE} introduced the concept of {\it  absolute ruin}  where the company is allowed to borrow money as long as the interest force of its debt is lower than the income rate. Albrecher et al. \cite{ACT, ACT1, AIZ} introduced the idea of {\it randomized observations} in order to monitor the event of ruin at random intervals as the company balances its books on a periodic basis.  Albrecher et al. \cite{AGS} introduced the Gamma-Omega model which assumes that a company will stop doing business whenever its surplus is negative with probability $\omega(x)\ud t$ within $\ud t$ time units, where $\omega(x)$ is a positive function for $x\le 0$.  Dassios and  Wu \cite{DW} introduced the concept of {\it Parisian ruin},   where they consider the application of an implementation delay in the recognition of an insurer's capital insufficiency. More precisely, they assume that ruin occurs if  the amount of  time that the surplus of company is negative is longer than a  than a deterministic time. It is worth pointing out that this definition of ruin is referred to as Parisian ruin due to its ties with Parisian options (see Chesney et al. \cite{CJY}). In \cite{DW}, the analysis of the probability of Parisian ruin is done in the context of the classical Cram\'er--Lundberg model. More recently, Landriault et al. \cite{LRZ-0,LRZ} and Loeffen et al. \cite{LCP} considered the idea of Parisian ruin with respectively a stochastic implementation delay and a deterministic implementation delay, but in the more general setup of L\'evy insurance risk models. In \cite{LRZ-0}, the authors assume that the deterministic delay is replaced by a stochastic grace period with a pre-specified distribution, but they restrict themselves to the study of a L\'evy insurance risk process with paths of bounded variation; explicit results are obtained in the case the delay is exponentially distributed. The unbounded variation case was recently studied by Baurdoux et al. \cite{BPPR}, using excursion theory arguments for L\'evy processes.  The model with deterministic delays has also been studied in the L\'evy setup by Czarna and Palmowski \cite{CP} and by Czarna \cite{C}.

In all these models, the authors distinguish between {\it being ruined} and {\it going out of business}, where the probability of {\it going out of business} is a functional of the level of negative surplus. In this paper, we propose a new  definition of {\it going out of business} which in  some sense is motivated by the notion of  Parisian ruin. In other words instead of being interested on the length of the negative excursions of the surplus, we are interested on their depths. Roughly speaking, for a  surplus process we mark each negative excursion with an independent random variable that represents the {\it severity of debt} of the surplus in the excursion. We declare bankruptcy  if a negative excursion of the surplus  exceeds its mark or its severity of debt.  Exactly the same technique allows us to deal with a marking that depends on the current excursion, as for instance, one could think in a more realistic model where the level of debt allowed to the company is a functional of the behaviour of the surplus prior to the last ruin; nevertheless, we will not provide details about this case because the formulas become cumbersome, and we are more interested in the techniques here used.

Originally motivated by  pricing  American claims, Gerber and Shiu \cite{GS, GS1} introduced in risk theory a function that jointly penalizes the present value of the time of ruin, the surplus before ruin and the deficit after ruin for Cram\'er--Lundberg-type processes.
Since then this expected discounted penalty  function, now known as the Gerber--Shiu function, has been deeply studied. Recently, Biffis and Kyprianou \cite{BK} characterized a generalized version of this function in the setting of  L\'evy insurance risk processes.

In this paper, we study the Gerber-Shiu function for our notion of bankruptcy for general L\'evy insurance risk processes.  Our approach is based on fluctuation theory for a process of this type (spectrally negative L\'evy processes) and, under the assumption that the process posses unbounded variation paths,  on a new description  of its entrance law of the  excursion measure away from zero  due to the authors, see \cite{cem}. All our results are given in terms of the so called scale functions.

The rest of the paper is organized as follows. In the remainder of Section 1, we introduce  L\'evy insurance risk processes and their associated scale functions and we state some well-known fluctuation identities that will be useful for the sequel. We also introduce, formally speaking,  our notion of bankruptcy  in terms of the excursions away from $0$ of the L\'evy insurance risk process and we provide the main results of this paper when the surplus starts at $0$.  Firstly, we study the bounded variation case where such definition is   easier to explain since  the process spends a positive amount of time above zero a.s. We finish Section 1 with the description of the unbounded variation case using It\^o's excursion theory. It is important to note that we  should distinguish between both cases since in the bounded variation case in each finite interval there are finitely many excursions away from 0 for the L\'evy insurance risk process, opposite to the unbounded variation where there are uncountable many, and thus the results in \cite{cem} are needed.  Section 2 is devoted to the proofs of the main results.  In both cases,  the  arguments are based in the same line of reasoning that can be roughly described as follows: the time of bankruptcy occurs in the first excursion which is deep enough, which can be tought as a defective excursion, this excursion is independent of the previous ones, the number of non-defective excursions before the  first defective follows a geometric distribution (in the bounded variation case) and in the local time scale this excursion appears at an exponential time. We conclude this manuscript, Section 3, with a brief description of how to extend our results to the case when the surplus starts at any positive value.

\subsection{L\'evy insurance risk processes}

Let $X=(X_t, t\geq 0)$ be a real-valued L\'evy process defined on a  probability space $(\Omega, \mathcal{F}, \p)$.  For $x\in \R$ denote by $\p_x$ the law of the L\'evy process $X$ when it is started at $x$ and write for convenience  $\p$ in place of $\p_0$. Accordingly, we shall write $\e_x$ and $\e$ for the associated expectation operators. In this paper we shall assume throughout that $X$ is \textit{spectrally negative} meaning that it has no positive jumps and that it is not the negative of a subordinator. It is well known that the law of $X$ is  determined by its Laplace exponent $\psi:[0,\infty) \to \R$, which is defined as follows
\[
\e\Big[{\rm e}^{\lambda X_t}\Big]=:{\rm e}^{\psi(\lambda)t}, \qquad t, \lambda\ge 0.
\]
The Laplace exponent $\psi$ has a  L\'evy--Khintchine representation, i.e.
\begin{equation}
\psi(\lambda)=\gamma\lambda+\frac{\sigma^2}{2}\lambda^2+\int_{(-\infty,0)}\big({\rm e}^{\lambda x}-1-\lambda x\mathbf{1}_{\{x>-1\}}\big)\Pi(\ud x),\notag
\end{equation}
where $\gamma\in \R$, $\sigma^2\ge 0$ and $\Pi$ is a measure with support on $(-\infty, 0)$ which is known as  the L\'evy measure of $X$ and satisfies
\[
\int_{(-\infty,0)}(1\land x^2)\Pi(\ud x)<\infty.
\]

Another important fact about this class of processes is that $X$ has paths of bounded variation if and only if $\sigma^2=0$ and $\int_{(0, 1)} x\Pi(\mathrm{d}x)<\infty$. In this case $X$ can be written as follows
\begin{equation}
X_t=ct-S_t, \,\,\qquad t\geq 0,
\label{BVSNLP}
\end{equation}
where $c=\gamma+\int_{(-1, 0)} x\Pi(\mathrm{d}x)$ and $(S_t,t\geq0)$ is a driftless subordinator. Note that  necessarily $c>0$, since we have ruled out the case that $X$ has monotone paths. In this case, its Laplace exponent can be written as follows
\begin{equation*}
\psi(\lambda)= \log \mathbb{E} \left[ \mathrm{e}^{\lambda X_1} \right] = c \lambda-\int_{(0,\infty)}\big(1- {\rm e}^{-\lambda x}\big)\Pi(\ud x).\notag
\end{equation*}
The reader is referred to monographs of Bertoin \cite{B} and Kyprianou \cite{K} for a complete introduction to the theory of L\'evy processes.\\

A key element, which is currently used,  of the forthcoming analysis relies on the theory of the so-called scale functions for spectrally negative L\'evy processes. We therefore devote some paragraphs in this section reminding the reader  some of their fundamental properties. For
each $\theta\geq0$, we define  $W^{(\theta)}:\R\to [0, \infty),$ such that $W^{(\theta)}(x)=0$ for all $x<0$ and on $(0,\infty)$ is the unique continuous function with Laplace transform
\begin{eqnarray}
\int^{\infty}_0\mathrm{e}^{-\lambda x}W^{(\theta)}(x){\rm d}x=\frac1{\psi(\lambda)-\theta},
\qquad \lambda>\Phi(\theta),\notag
\end{eqnarray}
where $ \Phi(\theta) = \sup\{\lambda \geq 0: \psi(\lambda) = \theta\}$ which is well defined and finite for all $\theta\geq 0$, since $\psi$ is a strictly convex function satisfying $\psi(0) = 0$ and $\psi(\infty) = \infty$. For convenience, we write $W$ instead of $W^{(0)}$. Associated to the functions $W^{(\theta)}$ are the functions $Z^{(\theta)}:\R\to[1,\infty)$ defined by
\[
Z^{(\theta)}(x)=1+\theta\int_0^x W^{(\theta)} (y)\ud y,\qquad \theta\ge 0.
\]
Together, the functions $W^{(\theta)}$ and $Z^{(\theta)}$ are collectively known as $\theta$-scale functions and predominantly appear in almost all fluctuations identities for spectrally negative L\'evy processes.

When $X$ has paths of bounded variation, without further assumptions, it can only be said that the function $W^{(\theta)}$ is almost everywhere differentiable on $(0,\infty)$. However, in the case that $X$ has paths of unbounded variation, $W^{(\theta)}$ is continuously differentiable on $(0,\infty)$; cf. Chapter 8 in \cite{K}. Throughout this text we shall write $W^{(\theta)\prime}$ to mean the well-defined derivative in the case of unbounded variation paths and a version of the density of $W^{(\theta)}$ with respect to Lebesgue measure in the case of bounded variation paths. This should cause no confusion as, in the latter case, $W^{(\theta)\prime}$ will only appear inside  Lebesgue integrals.

The theorem below is a collection of known fluctuation identities which will be used along this work. See for instance Chapter 8 of \cite{K} for proofs and the origin of these identities.
\begin{theorem}\label{fi}
Let $X$ be a spectrally negative L\'evy process and 
\begin{equation}
\tau_a^+=\inf\{t>0:X_t>a\}\qquad\text{and}\qquad\tau_0^-=\inf\{t>0:X_t<0\}.\notag
\end{equation}
\begin{itemize}
\item[(i)] For $\theta\geq 0$ and $x\leq a$
\begin{equation}\label{fi2}
\mathbb{E}_x\Big[e^{-\theta\tau_a^+}\mathbf{1}_{\{\tau_a^+<\infty\}}\Big]=e^{-\Phi(\theta)(a-x)}\qquad\textrm{and}\qquad\mathbb{E}_x\Big[e^{-\theta\tau_a^+}\mathbf{1}_{\{\tau_a^+\le \tau_0^-\}}\Big]=\frac{W^{(\theta)}(x)}{W^{(\theta)}(a)}.
\end{equation}
\item[(ii)] Let $a>0$, $x\in (0,a],q\geq 0$ and $f,g$ be positive, bounded measurable functions. Then
\begin{align}\label{fi1}
\mathbb{E}_x\Big[e^{-\theta\tau_0^-}&f(X_{\tau_0^-})g(X_{\tau_{0-}^-})\mathbf{1}_{\{\tau_0^-<\tau_a^+\}}\Big]\notag\\
&=\frac{\sigma^2}{2}f(0)g(0)\mathcal{O}^{(\theta)}(a,x)+\int_0^a\int_{(-\infty,-y)}f(y+u)g(y)\mathcal{W}^{(\theta)}(a,x,y)\Pi(\ud u) \ud y.
\end{align}
where
\begin{align*}
\mathcal{W}^{(\theta)}(a,x,y)&:=\frac{W^{(\theta)}(x)W^{(\theta)}(a-y)}{W^{(\theta)}(a)}-W^{(\theta)}(x-y), \\
\mathcal{O}^{(\theta)}(a,x)&:=W^{(\theta)\prime}(x)-W^{(\theta)}(x)\frac{W^{(\theta)\prime}(a)}{W^{(\theta)}(a)}.
\end{align*}
\end{itemize}
To simplify notation, we denote by $\mathcal{W}$ and $\mathcal{O}$ for $\mathcal{W}^{(0)}$ and $\mathcal{O}^{(0)}$, respectively. 
\end{theorem}
Another important quantity in what follows is the so-called $q$-resolvent of the L\'evy process $X$ which is defined by
 $$U_{q}(\ud y):=\int^{\infty}_{0}e^{-qt}\p(X_{t}\in \ud y)\ud t,\qquad y\in \mathbb{R}.$$
The $q$-resolvent $U_q$ is absolutely continuous with respect to Lebesgue measure and its density satisfies
\begin{equation}\label{resolvent-density}
u_{q}(y)=\Phi^{\prime}(q)e^{-\Phi(q)y}-W^{(q)}(-y),\qquad y\in \mathbb{R},
\end{equation} 
see for instance Corollary 8.9 in \cite{K} and exercise 2 in Chapter VII in \cite{B}. The latter identity  and the fact that scale functions take the value $0$ at $(-\infty,0)$ implies
\begin{equation}\label{resolvent-div}
\frac{u_{q}(x)}{u_{q}(0)}=e^{-x\Phi(q)},\qquad x>0.
\end{equation}
Throughout  this paper, we assume that the L\'evy insurance risk process X satisfies the net profit condition, i.e. 
\[
\e[X_1]=\psi^\prime(0+)> 0.
\] 
In this paper we are interested in studying Gerber-Shiu penalty function at time to bankruptcy in a parisian type  risk model. We split our study in two cases the bounded variation case and the unbounded variation case. The former is much simpler than the latter, as we will see later a full description of the excursion measure away from zero is necessary.  The description of the excursion measure away from zero  for spectrally negative L\'evy processes was  recently obtained by the authors  in \cite{cem}.

\subsection{Bounded variation case}
We first give a descriptive definition of the time
to bankruptcy, here denoted by  $T_B$, in the bounded variation case which can be explained in a very simple way.  In order to do so, we first recall that under this assumption the state 0 is irregular for $(-\infty, 0)$ (see for instance Corollary VII.5 in Bertoin \cite{B}), implying that a typical excursion of the process $X$ away from 0 behaves as follows: the excursion starts at 0 from above and stays in the positive half line for some period of time before crossing the state 0 by a jump and then the excursion follows until it reach 0 continuously. For our purpose, we are interested in the negative part of such excursions. 

Let $Y$ be a non-negative random variable which is independent of the L\'evy insurance risk process $X$, that represents the severity of debt at bankruptcy in this model. We assume that each negative  excursion away  from zero is accompanied by an independent copy of $Y$. We will refer to each of such random variables as the
severity of debt in the excursion. We say that bankruptcy occurs at the first time that any negative excursion is below the level $-Y$. Since $X$ has paths of bounded variation, the number of negative excursions before  bankruptcy occurs is  finite a.s. More precisely, let $(Y_k)_{k\geq 1}$ be a sequence of iid copies of the r.v. $Y$. Also, let us define, recursively, 
 two sequences
of stopping times $(\tau^{-,k}_{0})_{k\geq 1}$ and $(\tau^{+,k}_{0})_{k\geq 1}$ as follows. Set $\tau_0^{+,0}=0$ and for  $k\geq 1$, let
 \[
\tau_{0}^{-,k}=\inf\{t>\tau_{0}^{+,k-1}:X_t<0\}\qquad\textrm{and}\qquad \tau_{0}^{+,k}=\inf\{t>\tau_{0}^{-,k}:X_t>0\},
\]
which represent the first time, after $\tau_{0}^{+,k-1}$, that the process $X$ enters $(-\infty,0)$ and  the first time, after $\tau_{0}^{-,k}$, that $X$ enters $(0,\infty)$, respectively. Finally, we introduce the following sequence of stopping times $(\tau_{-Y_k}^{-,k})_{k\geq 1}$ that will be helpful in our definition of time to bankruptcy. Each element of the  latter is defined   as follows; for each $k\geq 1$ 
\begin{equation}
\tau_{-Y_k}^{-, k}=\inf\{t\geq 0 :X_{\tau_{0}^{-, k}+t}<-Y_k\},
\end{equation}
with the customary assumption that $\inf\{\emptyset\}=\infty.$
Thus, the time to bankruptcy $T_B$ is defined as follows
\[
T_B=\tau_{0}^{-, \kappa_B}+\tau_{-Y_{\kappa_B}}^{-, \kappa_B},\qquad\textrm{ where }\qquad
\kappa_B=\inf\{k\geq 1:\tau_{-Y_k}^{-,k}<\tau_{0}^{+,k}\}.\notag
\]
 Notice that $\kappa_B$ is the number of excursions necessary to observe a defective excursion for the first time.

As we mentioned before, we are interested in computing the Gerber-Shiu expected discounted penalty function for this new definition of bankruptcy. In order to do so, we will consider a measurable function $f:[0,\infty)^2\to[0,\infty)^2$, a constant $b>0$ and  a discounted force of interest $q\geq 0$. The penalty function, when the initial revenue of the insurance company is $x\ge 0$, is given by
\begin{equation}
\phi_f(x,q,b)=\e_x\left[e^{-qT_B}f(X_{T_B-},X_{T_B})\mathbf{1}_{\{T_B<\tau_b^+\}}\right],\notag
\end{equation}
where $\tau_b^+=\inf\{t>0:X_t>b\}$.

The Gerber-Shiu expected discounted penalty function for the L\'evy insurance risk process, in the bounded variation case,  at time $T_{B}$ is given by the following result.  To simplify notation, we introduce for $\theta, b>0$ and $x\in (-\infty,b)$,  
\[
\mathcal{H}^{(\theta)}(b, x)=\frac{W^{(\theta)}(b-x)}{W^{(\theta)}(b)}-e^{\Phi(\theta)b}\frac{W^{(\theta)}(-x)}{W^{(\theta)}(b)}.
\]
It is important to note that for  $x\in(0,b)$, we have
\[
\mathcal{H}^{(\theta)}(b, x)=\frac{W^{(\theta)}(b-x)}{W^{(\theta)}(b)}.
\]

\begin{theorem}\label{bvcase} 
For any $q\geq0$, we have
\begin{equation*}
\phi_f(0,q,b)=\frac{\mathcal{A}_f(q,b)+\mathcal{B}_f(q,b)}{\displaystyle\frac{1}{W^{(q)}(0+)}-\displaystyle \e\left[\int_0^{\infty}\mathcal{H}^{(q)}(b, y)\int_{(-y-Y,-y)}\frac{W^{(q)}(y+x+Y)}{W^{(q)}(Y)}\Pi(\ud x)\ud y\right]},
\end{equation*}
where
\begin{align}
\mathcal{A}_f(q,b)&:= \e\left[\int_0^{b} \mathcal{H}^{(q)}(b, y)\int_{(-\infty,-y-Y)}f(y,x+y)\Pi(\ud x)\ud y\right],\label{event21}\\
\mathcal{B}_f(q,b)&:=\e\bigg[\int_0^{b}\mathcal{H}^{(q)}(b, y)\int_{(-y-Y, -y)}\int_0^Y\int_{(-\infty,-z)}f(z-Y,u+z-Y)\label{eqb}\\
&\hspace{7cm}\times \mathcal{W}^{(q)}(Y, y+x+Y, z) \Pi(\ud u)\ud z\Pi(\ud x)\ud y\bigg] \notag.
\end{align}
\end{theorem}
It is important to note that denominator in the Gerber-Shiu functional in the above result is always positive. As we will see in the proof of this result, the expected value in the denominator is smaller than  $(W^{(q)}(0+))^{-1}$.

\subsection{Unbounded variation case, an approach using Ito's excursion theory}

Now, we describe the time to bankruptcy when the L\'evy insurance risk process $X$ has paths of unbounded variation.   Recall that under this assumption,  the state $0$ is regular for  itself, that is to say the process $X$ returns to the origin at arbitrarily small times a.s. We also recall that  this occurs if and only if  either $\sigma^{2}>0$ or  
\[
\int_{(-\infty,0)}(1\land |x|)\Pi(\ud x)=\infty.
\]
The previous comment suggest that a more sophisticated technique is needed since the behaviour of a typical excursion is not as simple as in the bounded variation case. Moreover, we may have an infinite number of excursion on small periods of time. Therefore a full description of the   excursions away from 0 of $X$ is needed in order to introduce  our definition of  bankruptcy. Our methodology is largely base on the recent paper of the authors in \cite{cem} where  excursion theory  for spectrally negative L\'evy processes away from 0 is studied.

General theory of Markov process, see e.g. Chapter IV in \cite{B}, ensures that there exists a local time at zero for $X,$ here denoted by  $(L_{t}, t\geq0).$  The right-continuous inverse of the  local time  $L$ is defined as follows
\begin{equation}\label{ilt}
L^{-1}_{t}=\inf\{s>0: L_{s}>t\},\qquad t\geq 0,
\end{equation}
which is known to be a, possibly killed,  subordinator. Let $\mathcal{E}$ be the space of real valued c\'adl\'ag paths  and $\Upsilon$ be an isolated state. The main result in It\^{o}'s excursion theory, as described in~\cite{blumenthal-exc} or \cite{fitzsimmons-getoor-exc}, ensures that the point process $(e_{t}, 0<t\leq L_{\infty})$ on $\mathcal{E}\cup\{\Upsilon\}$ defined by
\begin{equation}\label{excdef}
e_{t}=\begin{cases}
X_{\tau_{t-}+s}, 0\leq s\leq L^{-1}_{t}-L^{-1}_{t-}, & \text{ if }  \quad L^{-1}_{t}-L^{-1}_{t-}>0,\\
\Upsilon, & \text{ if }  \quad L^{-1}_{t}-L^{-1}_{t-}=0,
\end{cases}
\end{equation}
is a Poisson point process with intensity measure, say $\mathbf{n},$ stopped at the first point $e\in \mathcal{E}$ whose lifetime is infinite. The measure $\mathbf{n}$ is also called the excursion measure away from $0$ for $X.$ For further details on excursion theory  for spectrally negative L\'evy processes away from 0 and the description of the excursion measure $\mathbf{n}$, the reader is  referred to \cite{cem}.

 In order to describe the time to bankruptcy for $X$, we will mark each excursion away from zero of the L\'evy insurance risk process,  in other words, the excursion at local time $s$, say $e_{s},$ will be marked by a random variable $Y_{s}$, independent of $X$,  which has the same law as the r.v. $Y$, and thus bankruptcy will occur if this excursion is defective, which in this setting is taken one such that
\[
\inf_{0<u<\zeta}e_{s}(u)<-Y_{s}<0,\qquad \textrm{ or whenever }\qquad\inf_{0<u<\zeta}e_{s}(u)=0,\qquad\text{and} \quad\zeta>\bf{e}_{\lambda},
 \]
 where $\bf{e}_{\lambda}$ is an exponential r.v. of parameter $\lambda>0$.

 \par Observe that the latter possibility covers the case where bankruptcy occurs by creeping, that is the surplus hits the level $0$ from above, which when the excursion is long enough  (longer than an independent exponential r.v. of parameter $\lambda>0$) can be interpreted as a ruin due to an incorrect management of the insurance company and hence the bankrupt can not be avoided. This event has a strictly positive probability if and only if the surplus  creeps downwards, which  is equivalent to have a non-zero Brownian term in the Laplace exponent of the process $X$, i.e. $\sigma>0$. {\color{blue}}  
  
  The Gerber-Shiu expected discounted penalty function for the L\'evy insurance risk process, in the unbounded variation case,  at time $T_{B}$ is given by the following result. For the precise definition of $T_{B}$ in terms of excursions see the beginning of the Proof of the next theorem (see Section 2.2).
\begin{theorem}\label{ubvcase}
For any $q\geq0$ we have
\begin{align}
\phi_f(0,q,b)=\frac{\displaystyle \mathcal{A}_f(q,b)+\mathcal{B}_f(q,b)+\frac{\sigma^2}{2}\Big(\mathcal{U}_f(q,b)+\mathcal{C}_f(q)+f(0,0)\mathcal{J}(\lambda,q,b)\Big)}{\displaystyle \frac{e^{\Phi(q)b}}{W^{(q)}(b)}+\mathcal{D}(q,b)+\mathcal{E}(q,b)+\frac{\sigma^2}{2}\Big(\mathcal{F}(q)+\mathcal{J}(\lambda, q,b)\Big)}.\notag
\end{align}
where $\mathcal{A}_f(q,b)$ and $\mathcal{B}_f(q,\lambda,b)$ are defined as in Theorem 2 and 
\begin{align}
\mathcal{D}(q,b)&=\e\left[\int^{b}_{0}\mathcal{H}^{(q)}(b, x)\int_{(-\infty, -x-Y)}e^{\Phi(q)(x+y)}\Pi(\ud y)\ud x \right],\label{event2}\\
\mathcal{E}(q,b)&=\e\left[\int^{b}_{0}\mathcal{H}^{(q)}(b, x)\int_{(-x-Y,-x)}\left(e^{\Phi(q)(x+y)}-\frac{W^{(q)}(x+y+Y)}{W^{(q)}(Y)}\right)\Pi(\ud y)\ud x\right],\label{eqe}\\
\mathcal{F}(q)&=-\frac{\sigma^2}{2}\e\left[e^{-\Phi(q)Y}\frac{\partial}{\partial x}\mathcal{O}^{(q)}(Y, Y)\right]+\e\left[\int_0^Y\int_{(-\infty,-y)}e^{\Phi(q)(y+z-Y)}\mathcal{O}^{(q)}(Y, Y-y)\Pi(\ud z)\ud y \right].\notag\\
\mathcal{J}(\lambda, q,b)&= \int_{-\infty}^{b}\left((\lambda+q)\mathcal{H}^{(\lambda+q)}(b, y)-q\mathcal{H}^{(q)}(b, y)\right)\mathcal{O}(b,y)\ud y,\label{event1}
\end{align}
\begin{align}
\mathcal{C}_f(q)&=-\frac{\sigma^2}{2}\e\left[f(-Y,-Y)\frac{\partial}{\partial x}\mathcal{O}^{(q)}(Y,Y)\right]\notag\\
&\hspace{2cm}+\e\left[\int_0^Y\int_{(-\infty,-y)}f(y-Y,y+z-Y)\mathcal{O}^{(q)}(Y,Y-y)\Pi(\ud z)\ud y\right],\label{eqlem51}\\
\mathcal{U}_f(q,b)&=\e\left[\int_{0}^{b} \mathcal{H}^{(q)}(b,x) \int_{(-x-Y,-x)}\mathcal{O}^{(q)}(Y,Y+x+y)f(-Y, -Y)\Pi(\ud y)\ud x\right].\label{equ}
\end{align}
\end{theorem}

\section{Proofs:}
\subsection{Proof of Theorem \ref{bvcase} (bounded variation case)}
\begin{proof}

We are interested in computing
\begin{equation}\label{ott}
\mathbb{E}\left[e^{-q T_B}f(X_{T_B-},X_{T_B})\mathbf{1}_{\{T_B<\tau_b^+\}}\right].
\end{equation}
A crucial point in our analysis is that $0$ is irregular for $(-\infty,0)$ for $X$ on account of it having paths of bounded variation and moreover that $X$ does not creep downwards. This means that each excursion of $X$ away from $0$, under $\mathbb{P}$,  consists of a copy of $(X_t, t\leq \tau^-_0)$ issued from $0$ and 
the excursion continues from the time $\tau^-_0$ as an independent copy of $(X_t, t\leq \tau^+_0)$ issued from $X_{\tau^{-}_{0}}.$ We will denote a typical excursion by $\xi=(\xi_{s}, s\geq 0)$ and by $\zeta$ its lifetime, $\zeta=\inf\{t>0: \xi_{t}=0\}.$ We will be considering excursions that do not exceed the level $b$ from above.

Here, we  express (\ref{ott}) in terms of the excursions of the process $X$ whose heights are lower than  $b$ and its depths are smaller than independent random variables with the same distribution as $Y$  and, subsequently, the first excursion whose height  is bigger than $b$ or its depth is bigger than an independent random variable with the same law as $Y$.   More precisely, let  $\xi^{(i)}=(\xi_s^{(i)}, 0\le s\le \ell_i)$ be the  $i$-th excursion  of $X$ away from $0$ confined to the interval $[-Y_i,b]$, here $\ell_i$ denotes the length of the excursion $\xi^i$. Similarly, let $\xi^*=(\xi^*_s, 0\le s\le\ell^*)$  be the first excursion  of $X$ away from $0$ that exits the interval $[-Y^*,b]$, where $Y^*$ is an independent copy of $Y$, and $\ell_*$ its length. 
 
 From the strong Markov property, it is clear that the random variables $e^{-q\ell_i}$  are i.i.d. and  independent of $e^{-q\ell^*}f(\xi^*_{\ell^*-},\xi^*_{\ell^*})1_{\{\ell^*<\infty\}}$.  Define 
 $$
 E=\left\{\sup_{t\leq \zeta} \xi_t\leq b,  \inf_{t\leq \zeta} \xi_t \geq -Y\right\},
 $$ and
$p=\p(E)$. A standard description of the excursions $(\xi^{(i)})_{i\ge 1}$ dictates that the number of finite excursions is distributed according to an independent geometric  random variable, say $G_p$, (supported on $\{0,1,2,\ldots\}$) with parameter $p$, the random variables $e^{-q\ell_i}$ are equal in distribution to $e^{-q\zeta}$ under the conditional law $\p(\cdot|E)$ and the random variable $e^{-q\ell^*}f(\xi^*_{\ell^*-},\xi^*_{\ell^*})\mathbf{1}_{\{\ell^*<\infty\}}\mathbf{1}_{\{\xi^*_{\ell^*}<-Y\}}$ is equal in distribution $e^{-q\tau_{-Y}^-}f(X_{\tau_{-Y}^- -},X_{\tau_{-Y}^-})\mathbf{1}_{\{\tau_{-Y}^-<\tau_b^+\}}$ but now under the conditional law $\p(\cdot|E^c)$.
It now follows that
\begin{align}
\mathbb{E}\bigg[e^{-q T_B}f(X_{T_B-},X_{T_B})\mathbf{1}_{\{T_B<\tau_b^+\}}\bigg]&=\e\left[\prod_{i=0}^{G_p}e^{-q\ell_i}e^{-q\ell^*}f(\xi^*_{\ell^*-},\xi^*_{\ell^*})\mathbf{1}_{\{\ell^*<\infty\}}\mathbf{1}_{\{\xi^*_{\ell^*}<-Y\}}\right]\notag\\
&=\mathbb{E}\left[\e\left[e^{-q\ell_1}\right]^{G_p}\right]\e\left[e^{-q\ell^*}f(\xi^*_{\ell^{*}-},\xi^*_{\ell^{*}})\mathbf{1}_{\{\ell^*<\infty\}}\mathbf{1}_{\{\xi^*_{\ell^*}<-Y\}}\right].
\label{putin}
\end{align} 
Recalling that the generating function of the independent geometric random variable $G_p$ satisfies,
\[
F(s)=\frac{\hat{p}}{1-sp},\qquad |s|<\frac{1}{p},
\]
where $\hat{p}=1-p$, we get that  the first term of the right-hand side of the above identity satisfies
\begin{equation}
\begin{split}
\mathbb{E}\left[\e\left[e^{-q\ell_i}\right]^{G_p}\right]
&=\displaystyle\frac{\hat{p}}{1-p\e\left[e^{-q\ell_1}\right]}.
\end{split}
\label{pieces together}
\end{equation}
Moreover, again taking account of the remarks in the previous paragraph and applying the identities (\ref{fi2}) and (\ref{fi1}) in Theorem 1, we also have that
\begin{align*}
\e\left[e^{-q\ell_1}\right]&=\frac{1}{p}\e\left[e^{-q\tau_0^-}\mathbf{1}_{\{\tau_0^-<\tau_b^+,X_{\tau_0^-}>-Y\}}\e_{X_{\tau_0^-}}\left[e^{-q\tau_0^+},\tau_0^+<\tau_{-Y}^-\right]\right]\notag\\
&=\frac{1}{p}\e\left[e^{-q\tau_0^-}\mathbf{1}_{\{\tau_0^-<\tau_b^+, X_{\tau_0^-}>-Y\}}\e_{X_{\tau_0^-}+Y}\left[e^{-q\tau_Y^+},\tau_Y^+<\tau_{0}^-\right]\right]\notag\\
&=\frac{1}{p}\e\left[e^{-q\tau_0^-}\mathbf{1}_{\{\tau_0^-<\tau_b^+,X_{\tau_0^-}>-Y\}}\frac{W^{(q)}(X_{\tau_0^-}+Y)}{W^{(q)}(Y)}\right]\notag\\
&=\frac{W^{(q)}(0+)}{p}\e\left[\int_0^{b}\mathcal{H}^{(\theta)}(b, x)\int_{(-y-Y,-y)}\frac{W^{(q)}(y+x+Y)}{W^{(q)}(Y)}\Pi(\ud x)\ud y\right].
\end{align*}
 Hence putting all the pieces together in (\ref{pieces together}), we obtain
\begin{equation}
\mathbb{E}\left[\e\left[e^{-q\ell_i}\right]^{G_p}\right]=\frac{\hat{p}}{1-\displaystyle W^{(q)}(0+)\e\left[\int_0^{b}\mathcal{H}^{(\theta)}(b, x)\int_{(-y-Y, -y)}\frac{W^{(q)}(y+x+Y)}{W^{(q)}(Y)}\Pi(\ud x)\ud y\right]}.
\label{shitty}
\end{equation}

Next, we compute  the second term in the right-hand side of identity (\ref{putin}).  Recalling that $\ell_*$ is equal in law to $\tau_{-Y}^-\land\tau^+_b$ under $\p(\cdot|E^c)$  then, under the event that $\{\xi^*_{\ell^*}<-Y\}$, we have two  cases to consider. In the first case, the process $X$ jumps below 0  and the undershoot is bigger than $Y$, i.e. $-X_{\tau_0^-}>Y$. In the second case, the process $X$ jumps below $0$ and the undershoot is smaller than $Y$, and stays a positive period of time in the interval $(-Y,0)$  until it jumps below the level $-Y$. Therefore,
\begin{equation}\label{bigshit}
\begin{split}
\hat{p}\e\bigg[e^{-q\ell^*}f(\xi^*_{\ell^{*}-},\xi^*_{\ell^{*}})&\mathbf{1}_{\{\ell^*<\infty, \,\,\xi^*_{\ell^*}<-Y\}}\bigg]=\mathbb{E}\Big[e^{-q\tau_0^-}f(X_{\tau_0^--},X_{\tau_0^-})\mathbf{1}_{\{\tau_0^-<\tau_b^+,\,\, X_{\tau_0^-}<-Y\}}\Big]\\
&+\e\Big[e^{-q\tau_0^-}\mathbf{1}_{\{\tau_0^-<\tau_b^+, X_{\tau_0^-}>-Y\}}\e_{X_{\tau_0^-}}\Big[e^{-q\tau_{-Y}^-}f(X_{\tau_{-Y}^--},X_{\tau_{-Y}^-})\mathbf{1}_{\{\tau_{-Y}^-<\tau_0^+\}}\Big]\Big]\\
\end{split}
\end{equation}
Let us consider the first term  in the right-hand side of (\ref{bigshit}), then using (\ref{fi1}) we obtain 
\begin{align}\label{shitty2}
\mathbb{E}\Big[e^{-q\tau_0^-}&f(X_{\tau_0^--},X_{\tau_0^-})\mathbf{1}_{\{\tau_0^-<\tau_b^+, X_{\tau_0^-}<-Y\}}\Big]\notag\\
&=W^{(q)}(0+)\e\left[\int_0^{b}\mathcal{H}^{(q)}(b,y)\int_{(-\infty, -y-Y)} f(y,x+y)\Pi(\ud x)\ud y\right]=W^{(q)}(0+)\mathcal{A}_f(q,b)
\end{align}
Now for the second term, we use again  (\ref{fi1}) and observe the following
\begin{align}\label{shitty3}
  \frac{1}{W^{(q)}(0+)}&\e\bigg[e^{-q\tau_0^-}\mathbf{1}_{\{\tau_0^-<\tau_b^+, X_{\tau_0^-}>-Y\}}\e_{X_{\tau_0^-}}\Big[e^{-q\tau_{-Y}^-}f(X_{\tau_{-Y}^--},X_{\tau_{-Y}^-})\mathbf{1}_{\{\tau_{-Y}^-<\tau_0^+\}}\Big]\bigg]\notag\\
&\hspace{-1cm}=\e\left[\int_0^{b}\mathcal{H}^{(q)}(b,y)\int_{(-y-Y,-y)}\e_{y+x+Y}\left[e^{-q\tau_{0}^-}f(X_{\tau_{0}^--}-Y,X_{\tau_{0}^-}-Y)\mathbf{1}_{ \{\tau_{0}^-<\tau_Y^+\}}\right]\Pi(\ud x)\ud y\right]\notag\\
&\hspace{-1cm}=\e\Bigg[\int_0^{b}\mathcal{H}^{(q)}(b,y)\int_{(-y-Y, -y)}\int_0^Y\int_{(-\infty,-z)}f(z-Y,u+z-Y)\notag\\
&\hspace{5cm}\times\mathcal{W}^{(q)}(Y, y+x+Y, z) \Pi(\ud u)\ud z\Pi(\ud x)\ud y\Bigg]=\mathcal{B}_f(q,b).
\end{align}
Plugging (\ref{shitty2}) and (\ref{shitty3}) back into (\ref{bigshit}), we get 
\begin{align}\label{shitty4}
\mathbb{E}\Big[e^{-q\ell^*}f(\xi^*_{\ell^{*}-},\xi^*_{\ell^{*}})\mathbf{1}_{\{\ell^*<\infty, \,\,\xi^*_{\ell^*}<-Y\}}\Big]=\frac{W^{(q)}(0+)}{\hat{p}}\Big(\mathcal{A}_f(q,b)+\mathcal{B}_f(q,b)\Big).
\end{align}
Finally putting all  together, we obtain the desired identity.
\end{proof}

\subsection{Proof of Theorem \ref{ubvcase} (unbounded variation case)}
As we mention before, under the assumption of unbounded variation paths a description of the excursion measure  $\mathbf{n}$ is needed.  We recall that our  methodology and computations  are largely based on the recent development in \cite{cem},  where the  excursion theory away from 0 for spectrally negative L\'evy processes is treated and the measure $\mathbf{n}$ is described. For this reason, we use the same notation as in \cite{cem}.

In order to compute (\ref{ott}) in the unbounded variation case,   we split the path of the process at the  left extrema of the excursion away from $0$   where  bankruptcy occurs or has exceeded the level $b>0$. For simplicity, we denote  this time by   $g_{T_B}$, i.e.
\[
g_{T_B}=\sup\{t< T_B: X_t=0\},
\] 
and  we call {\it bankruptcy excursion} for the excursion where  bankruptcy occurs or has exceeded the level $b$. We  also denote by $\Theta_t$, the shift operator at time $t\ge 0$, and 
 by  $T^e_B$ for   the time when bankruptcy occurs in the bankruptcy excursion and observe that  we can write  the time to bankruptcy $T_B$, under the event $\{\tau_b^+>T_B\},$ as follows
\[
T_B=g_{T_B}+T_B\circ\Theta_{g_{T_B}}=g_{T_B}+T^e_B,
\]
implying that (\ref{ott}) satisfies
\begin{equation*}
\begin{split}
\mathbb{E}&\left[e^{-q T_B}f(X_{T_B-},X_{T_B})\mathbf{1}_{\{T_B<\tau_b^+\}}\right]=\e\left[e^{-q g_{T_B}}\left(e^{-q T^e_B}f( X_{T^e_B-}, X_{T^e_B})\mathbf{1}_{\left\{\tau_{b}^+\circ\Theta_{ g_{T_B}}>T_B\right\}}\right)\right],
\end{split}
\end{equation*}
where $\tau_b^{+}\circ\Theta_{g_{T_B}}$ denotes the time when the bankruptcy excursion passes above $b$ for the first time.

Now, we define the events $E_B=\{\textrm{Bankruptcy occurs}\}$ and $E=E_B\cup\{\tau_b^+<\zeta\}$. A simple application of the so-called Master's formula  at time $g_{T_B}$ (see for instance excursions straddling a terminal time in Chapter XII in Revuz-Yor \cite{RY})  implies  
\begin{align}\label{eqwexc}
\mathbb{E}\left[e^{-q T_B}f(X_{T_B-},X_{T_B})\mathbf{1}_{\{T_B<\tau_b^+\}}\right]=\e\Big[e^{-q g_{T_{B}}}\Big]\mathbf{n}\left(e^{-qT_B^e}f(X_{T_B^e-}, X_{T_B^e}) \mathbf{1}_{\{\tau_{b}^+>\zeta\}}\Big | E\right).
\end{align}
In order to compute the first expectation of the right hand-side of the above identity, we first recall some facts of excursion theory.  Let $(e_{t}, t\geq 0)$ be the point process of excursions away from $0$ of $X$ defined as in (\ref{excdef}) and $V:=\inf\{t>0: e_{t}\in E\}.$ It is well known that $(e_{t}, t<V)$ is independent of $(V,e_{V})$ (see for instance Proposition 0.2 in  \cite{B}). The former is a Poisson point process with characteristic measure $\mathbf{n}(\cdot \cap E^c)$ and ${V}$ follows an exponential distribution with parameter $\mathbf{n}(E).$ Moreover, we have that $g_{T_B}=\sum_{s<V}\zeta(e_{s})$, where $\zeta(e_{s})$ denotes the lifetime of the excursion $e_{s}$.  Therefore,  the exponential formula  for Poisson point processes (see for instance Section 0.5 in \cite{B} or Proposition 1.12 in Chapter XII in \cite{RY}) and the independence between $(e_{t}, t<V)$ and  $(V,e_{V})$ implies
\[
\begin{split}
\e\Big[e^{-q g_{T_A}}\Big]&=\e\left[\exp\left\{-q \sum_{s<V}\zeta(e_{s})\right\}\right]\\
&=\mathbf{n}(E)\int_0^\infty e^{-s(\mathbf{n}(E)+\mathbf{n}(1-e^{-q\zeta}, E^c))}\ud s\\
&=\frac{\mathbf{n}(E_B\cup\{\tau_b^+<\zeta\})}{\mathbf{n}(E_B\cup\{\tau_b^+<\zeta\})+\mathbf{n}\left(\mathbf{e}_q<\zeta, E^c_B,\tau_b^+>\zeta\right)}\\
&=\frac{\mathbf{n}(E_B\cup\{\tau_b^+<\zeta\})}{\mathbf{n}(\{\mathbf{e}_q<\zeta\}\cup\{\tau_b^+<\zeta\})+\mathbf{n}\left(\mathbf{e}_q>\zeta, E_B,\tau_b^+>\zeta\right)}.\\
\end{split}
\]
In order to simplify the above expression, we recall that  $L$ denotes the local time at zero of the process $X$ and  observe that $L_{\tau^+_b\land \mathbf{e}_q}$ is the first time where the Poisson point process $(e_{t}, t\geq 0)$ enters the set $\{\epsilon: h(\epsilon)>b \textrm{ or } \zeta(\epsilon)>\mathbf{e}_q\}$, where $h$ denotes the height of the excursion. In other words,  $L_{\tau^+_b\land \mathbf{e}_q}$ is the first time where a excursion goes above the level $b$ or its length is bigger than $\mathbf{e}_q$. Again using Proposition 0.2 in  \cite{B}, we deduce that  $L_{\tau^+_b\land \mathbf{e}_q}$ is exponentially distributed with parameter
\begin{equation}
\mathbf{n}\Big(\{\zeta>\mathbf{e}_{q}\}\cup \{\tau_{b}^+<\zeta\}\Big)=\e\Big[L_{\tau_{b}^+\wedge\mathbf{e}_q}\Big]^{-1}.\notag
\end{equation}
On the other hand, using  Fubinni's Theorem,   Lemma V.11 in \cite{B}, (\ref{resolvent-density}) and   (\ref{resolvent-div}),  we deduce
\begin{equation}\label{localtb}
\e\Big[L_{\tau_{b}^+\wedge\mathbf{e}_q}\Big]=\mathbb{E}\left[\int_0^{\tau_b^+}e^{-qs}\ud L_s\right]=u^q(0)-\frac{u^q(b)u^q(-b)}{u^q(0)}=e^{-\Phi(q)b}W^{(q)}(b).
\end{equation}
Next, using that $\mathbf{e}_q$ is an independent exponential r.v. with parameter $q$, we have 
\begin{equation}
\mathbf{n}\left(\mathbf{e}_q>\zeta, \tau_{b}^+>\zeta,E_B\right)=\mathbf{n}\left(e^{-q\zeta}, \tau_{b}^+>\zeta, E_B\right).\notag
\end{equation}
Hence putting all the pieces together, we deduce that (\ref{eqwexc}) can be written as follows,
\begin{equation}\label{eqwexc1}
\mathbb{E}\left[e^{-q T_B}f(X_{T_B-},X_{T_B})\mathbf{1}_{\{T_B<\tau_b^+\}}\right]=\frac{\mathbf{n}\left(e^{-qT_B^e}f(X_{T_B^e-}, X_{T_B^e}) \mathbf{1}_{\{\tau_{b}^+>\zeta\}}\right)}{\displaystyle \frac{e^{\Phi(q)b}}{W^{(q)}(b)}+\mathbf{n}\left(e^{-q\zeta}, \tau_{b}^+>\zeta, E_B\right)}.
\end{equation}
This implies that  our result can be computed if  we know how to describe the random variables 
\begin{equation}\label{rvundern}
(T_B^e, X_{T_B^e-}, X_{T_B^e}),
\end{equation}
under the excursion measure $\mathbf{n}.$  

Since the process $X$ is spectrally negative and possesses paths of unbounded variation then it is regular upwards and downwards, that is to say that $X$ enters the open upper and lower half line a.s. immediately  from 0, i.e. $\p(\tau_0^-=0)=\p(\tau_0^+=0)=1$ (see for instance Chapter 8 in \cite{K}). This implies that a typical excursion  where bankruptcy occurs has the following behaviour: it starts at zero and stays above zero until the first time where the excursion process crosses below zero continuously or by a jump. If the process hits zero continuously, the excursion ends, { and bankruptcy occurs when the length of this positive excursion is greater than an exponential random variable $\mathbf{e}_{\lambda}$ with rate $\lambda>0$ (i)}.  If the process crosses below  zero  by a jump, we have three possibilities, either the process jumps to a level which is below the value of the mark of the excursion (ii), or it jumps  above the value of the mark and then the process continues up to the time where it passes below the value of the mark continuously (iii), or by a jump (iv). In the case when the process $X$ has a non-zero Brownian component, we have an additional behaviour:  the excursion starts at zero and stays below zero until the first time that the process passes below the level of the mark either continuously or by a jump (v). 


The previous description suggest that we  should describe and excursion over the following five disjoint sets: 
\begin{itemize}
\item[(i)] $\displaystyle \{X_{\zeta-}=0, \mathbf{e}_{\lambda}<\zeta\},$ on which we have 
$$(T_B^e, X_{T_B^e-}, X_{T_B^e})=(\zeta, 0, 0);$$
\item[(ii)] $\displaystyle \{X_{\tau_{0}^--}>0, X_{\tau_{0}^-}\leq -Y,\},$ on which we have  
$$(T_B^e, X_{T_B^e-}, X_{T_B^e})=\Big(\tau_0^-, X_{\tau_0^--}, X_{\tau_0^-}\Big);$$
\item[(iii)] $\displaystyle \{X_{\tau_{0}^--}>0, 0>X_{\tau_{0}^-}>-Y, X_{\tau^{-}_{-Y}-}= -Y=X_{\tau^{-}_{-Y}}\},$ on which we have
$$(T_B^e, X_{T_B^e-}, X_{T_B^e})=\Big(\tau_{-Y}^-, -Y, -Y\Big);$$
\item[(iv)] $\displaystyle \{X_{\tau_{0}^--}>0, 0>X_{\tau_{0}^-}>-Y, X_{T_{B}-}> -Y>X_{T_{B}}\}$, on which we have
$$(T_B^e, X_{T_B^e-}, X_{T_B^e})=\Big(\tau_{-Y}^-, X_{\tau_{-Y}^--}, X_{\tau_{-Y}^-}\Big);$$
\item[(v)]  $ \{\sup_{s<\zeta}X_s\leq 0, X_{T_{B}}\leq { -Y}\},$ on which we have 
$$(T_B^e, X_{T_B^e-}, X_{T_B^e})=( \tau_{-Y}^-, X_{\tau_{-Y}^-}, X_{\tau_{-Y}^-}).$$
\end{itemize}
So, in order to conclude we will give explicit identities for the quantities involved in (\ref{eqwexc1}) under the five cases discussed above.

Our first Lemma computes the potential of the process $X$ under the excursion measure $\mathbf{n}$ until the first time that $X$ reaches the level $b>0$. Recall that for  $\theta, b>0$ and $x\in (-\infty,b)$,  
\[
\mathcal{H}^{(\theta)}(b, x)=\frac{W^{(\theta)}(b-x)}{W^{(\theta)}(b)}-e^{\Phi(\theta)b}\frac{W^{(\theta)}(-x)}{W^{(\theta)}(b)}.
\]
\begin{lemma}
For $q,b>0$ and $f:\mathbb{R}\to \mathbb{R}$ measurable it follows that
\begin{equation}\label{potentialkilled}
\mathbf{n}\left(\int_0^{\zeta\wedge \tau_{b}^+}e^{-qu}f(X_u)\ud u\right)=\int_{-\infty}^b f(y)\mathcal{H}^{(q)}(b, y)\ud y,
\end{equation}
\end{lemma}
\begin{proof}
Let $\mathcal{G}$ denote the set of left-end points of the excursion intervals from the state $0,$ and for $g\in\mathcal{G},$ we denote by $d_{g}$ the right-end point of the excursion interval starting at $g,$ that is $$d_{g}=\inf\{s>g: X_{s}=0\}.$$  We first  recall from   Theorem 2.7 in \cite{KKR} the following identity
\begin{equation}\label{eqinKKR}
\mathbb{E}\left[\int_{0}^{\tau^+_b}e^{-qt}f(X_t)\ud t\right]=\int_{-\infty}^bf(y)\left(e^{-\Phi(q)b}W^{(q)}(b-y)-W^{(q)}(-y)\right)\ud y,
\end{equation}
and, in particular, we observe
\[
\mathbb{E}\left[\int_{0}^{\tau^+_b}e^{-qt}f(X_t)\mathbf{1}_{\{X_{t}=0\}}\ud t\right]=\int_{-\infty}^bf(y)\mathbf{1}_{\{y=0\}}\left(e^{-\Phi(q)b}W^{(q)}(b-y)-W^{(q)}(-y)\right)\ud y=0.
\]
On the one hand,  by splitting   the interval $[0,\tau^+_b]$ into the excursion intervals and applying the Master's  formula or compensation formula (see for instance Section 0.5 in \cite{B} or Proposition 1.10 in Chapter XII in \cite{RY}) from excursion theory, we get 
\begin{align}
\mathbb{E}\bigg[\int_{0}^{\tau^+_b}e^{-qt}f(X_t)\ud t\bigg]&= \mathbb{E}\left[\int_{0}^{\tau^+_b}e^{-qt}f(X_t)\mathbf{1}_{\{X_{t}=0\}}\ud t\right]+\mathbb{E}\left[\int_{0}^{\tau^+_b}e^{-qt}f(X_t)\mathbf{1}_{\{X_{t}\neq 0\}}\ud t\right]\notag\\
&=\mathbb{E}\left[\sum_{g\in \mathcal{G}}\mathbf{1}_{\{g<\tau^+_b\}}\int_{g}^{d_{g}}e^{-qt}f(X_t)\ud t\right]\notag\\
&=\mathbb{E}\left[\sum_g \mathbf{1}_{\{g<\tau^+_b\}} e^{-qg}\left(\int_0^{\infty}e^{-qt}f(X_t)\mathbf{1}_{\{t<\zeta\wedge \tau^+_b\}}\ud t\right)\circ{}{\Theta}_{g}\right]\notag\\
&=\mathbb{E}\left[\int_0^{\tau^+_b}e^{-qs}dL_s\right]\mathbf{n}\left(\int_0^{\tau^+_b\wedge \zeta}e^{-qt}f(X_t)\ud t\right)\notag\\
&=e^{-\Phi(q)b}W^{(q)}(b)\mathbf{n}\left(\int_0^{\tau^+_b\wedge \zeta}e^{-qt}f(X_t)\ud t\right),\notag
\end{align}
where in the last equality we  used identity (\ref{localtb}), and   we recall that $\Theta_t$  denotes the shift operator at time $t$. This clearly implies the result.
\end{proof}

The information about the first event is given in the following Lemma. In particular, it provides the function $\mathcal{J}(\lambda, q, b) $ defined  in (\ref{event1}).
\begin{lemma}For $\lambda, q,b>0$, we have
\begin{equation*}
\begin{split}
\mathbf{n}&\left(e^{-q\zeta}f(X_{\zeta-},X_{\zeta})\mathbf{1}_{\{\mathbf{e}_{\lambda}<\zeta\}\cap E_1}\right)=\frac{\sigma^2}{2} f(0,0)\int_{-\infty}^b\left((\lambda+q)\mathcal{H}^{(\lambda +q)}(b,x)-q\mathcal{H}^{(q)}(b,x)\right)\mathcal{O}(b,x)\ud x,\notag\\
\end{split}
\end{equation*}
where 
\[
E_1=\left\{\ \inf_{0<s<\zeta}X_{s}=0,\ \sup_{0<s<\zeta}X_s<b,\ X_{\zeta-}=0\right\}.
\]

\end{lemma}

\begin{proof} We first note that a simple application of identity (\ref{fi1}), implies 
$$\p_{x}\left(\tau_0^-<\tau_{b}^+, X_{\tau_0^-}=0\right)=\mathcal{O}(b,x), \qquad x>0.$$
On the other hand,  using  integration by parts,  the Markov property under the excursion measure $\mathbf{n}$ (see  Section 4 in \cite{cem}), the previous identity and Lemma 1,  we deduce for any $\theta\geq0$,
\begin{equation*}
\begin{split}
&\hspace{-3cm}\mathbf{n}\left((1-e^{-\theta \zeta})\mathbf{1}_{E_1}\right)=\theta\, \mathbf{n}\left(\int^{\zeta\wedge \tau_{b}^+}_{0}e^{-\theta s}\,\p_{X_{s}}\left(\tau_0^-<\tau_{b}^+, X_{\tau_0^-}=0\right)\ud s\right)\\
&=\theta\,\frac{\sigma^2}{2} \mathbf{n}\left(\int^{\zeta\wedge \tau_{b}^+}_{0}e^{-\theta s}\mathcal{O}(b, X_s)\ud s\right)\\
&=\theta\frac{\sigma^2}{2} \int_{-\infty}^b\mathcal{O}(b,x)\mathcal{H}^{(\theta)}(b,x)\ud x.
\end{split}
\end{equation*}
Thus the result follows  by noting that under the  event $E_1$, $f(X_{\zeta-},X_{\zeta})=f(0,0)$, and 
\begin{align}
\mathbf{n}&\left(e^{-q\zeta}(1-e^{-\lambda \zeta})\mathbf{1}_{E_1}\right)=\mathbf{n}\left((1-e^{-(\lambda+q) \zeta})\mathbf{1}_{E_1}\right)-\mathbf{n}\left((1-e^{-q \zeta})\mathbf{1}_{E_1}\right).\notag
\end{align}
This completes the proof.
\end{proof}

In order to describe the event in case (ii),  the following Lemma is needed. In particular, we deduce the funtions $\mathcal{A}_f(q,b)$ in (\ref{event21})
 and $\mathcal{D}(q,b)$ in (\ref{event2})
 \begin{lemma}Let  $q, b>0$ and define
\begin{equation}\label{conj2}
E_2=\left\{ \sup_{u<\tau_0^-}X_u<b, \ \tau_0^-<\zeta\right\}.
\end{equation}
Then, we have 
\begin{align}
\mathbf{n}\left(e^{-q \tau_0^-}f(X_{\tau_0^--}, X_{\tau_0^-})\mathbf{1}_{ \{X_{\tau_0^-}<-Y\}\cap E_2}\right)&=\mathbb{E}\left[\int_{0}^{b} \mathcal{H}^{(q)}(b,x)\int_{(-\infty, -Y-x)}f(x, x+y)\Pi(\ud y)\ud x\right],\label{eqlem1} \\
  \mathbf{n}\Big(e^{-q\zeta}\mathbf{1}_{ \{X_{\tau_0^-}<-Y\}\cap E_2}\Big)&=\mathbb{E}\left[\int_{0}^{b} \mathcal{H}^{(q)}(b,x)\int_{(-\infty, -Y-x)}e^{\Phi(q)(x+y)}\Pi(\ud y)\ud x\right].\label{eqlem2}
\end{align}
\end{lemma}
\begin{proof}

We first prove identity (\ref{eqlem1}). Recall that $\Delta X_s$, denotes the jump of $X$ at time $s$, i.e. $\Delta X_s=X_s-X_{s-}$.
Since under $\mathbf{n}$,  the jumps of $X$  constitute a Poisson point process (see  Section 4 in \cite{cem}), an application of the Master's formula under $\mathbf{n}$ (see identity (14) in \cite{cem}) and Lemma 1 allow us to deduce, 
\begin{equation}\label{eqlem3}
\begin{split}
\mathbf{n}\Bigg(&e^{-q \tau_0^-}f(X_{\tau_0^--}, X_{\tau_0^-})\mathbf{1}_{ \{X_{\tau_0^-}<-Y\}\cap E_2}\Bigg)\\
&=\mathbf{n}\left(\sum_{0<s<\zeta}e^{-qs} f(X_{s-}, X_{s-}+\Delta X_{s})\mathbf{1}_{\left\{\inf_{u<s}X_{u}>0,\ \sup_{u<\tau_0^-}X_u<b, \ X_{s}<-Y\right\}}\right)\\
&=\mathbf{n}\left(\int^{\zeta}_{0}\ud s\,e^{-qs}\mathbf{1}_{\left\{\inf_{u<s}X_{u}>0, \ \sup_{u<\tau_0^-}X_u<b\right\}}\int_{(-\infty,0)}\Pi(\ud y)f(X_{s-}, X_{s-}+y)\mathbf{1}_{\{X_{s-}+y<-Y\}}\right)\\
&=\mathbf{n}\left(\int^{\zeta\wedge \tau_{b}^+}_{0}\ud s\, e^{-qs}\mathbf{1}_{\left\{X_{s-}>0\right\}}\int_{(-\infty,0)}\Pi(\ud y)f(X_{s-}, X_{s-}+y)1_{\{X_{s-}+y<-Y\}}\right)\\
&=\mathbb{E}\left[\int_{0}^{b} \mathcal{H}^{(q)}(b,x) \int_{{(-\infty, -Y-x)}}f(x, x+y)\Pi(\ud y)\ud x\right],\
\end{split}
\end{equation}
where in the third equality, we  used the fact that we can replace, under the excursion measure $\mathbf{n}$, the event $\left\{\inf_{u<s}X_{u}>0\right\}$ by $\left\{X_{s-}>0\right\}$ since the absence of positive jumps implies that once an excursion is below $0$ it ends continuosly.

We now prove identity (\ref{eqlem2}). We first apply the  strong Markov property at time $\tau_0^-$ under $\mathbf{n}$ and observe that for any $a>0$
\[
\mathbf{n}\Big(e^{-q\zeta}1_{ \{X_{\tau_0^-}<-Y\}\cap E_2}\Big)=\mathbf{n}\left(e^{-q\tau_0^-} \mathbf{1}_{ \{X_{\tau_0^-}<-Y\}\cap E_2} \e_{X_{\tau_0^{-}}}\Big[e^{-q\tau_0^+}\mathbf{1}_{\{\tau_0^+<\infty\}}\Big]\right).
\]
Hence using identities  (\ref{eqlem3}) and (\ref{fi2}), we deduce
\begin{align*}
\mathbf{n}\Big(e^{-q\zeta}1_{  \{X_{\tau_0^-}<-Y\}\cap E_2}\Big)&=\mathbb{E}\left[\int_{0}^{b} \mathcal{H}^{(q)}(b,x) \int_{{(-\infty, -Y-x)}}\e_{x+y}\Big[e^{-q\tau_0^+}\mathbf{1}_{\{\tau_0^+<\infty\}}\Big]\Pi(\ud y)\ud x\right]\\
&=\mathbb{E}\left[\int_{0}^{b} \mathcal{H}^{(q)}(b,x) \int_{{(-\infty, -Y-x)}}e^{\Phi(q)(x+y)}\Pi(\ud y)\ud x\right].
\end{align*}
The proof is now completed.
\end{proof}
Next Lemma describes  the cases  (iii) and (iv). In particular, we obtain the functions $\mathcal{B}_f(q,b), \mathcal{U}_f(q,b)$ and $\mathcal{E}(q,b)$ in (\ref{eqb}), (\ref{equ}) and (\ref{eqe}) respectively.
\begin{lemma} $i)$ For $ q, b>0$, we have
\begin{align}
&\hspace{-1cm}\mathbf{n}\left(e^{-q \tau_{-Y}^-}f(X_{\tau_{-Y}^--}, X_{\tau_{-Y}^-},)\mathbf{1}_{\{\tau_0^-< \tau_{-Y}^-<\zeta<\tau_b^+\}}\right)\notag\\
&\hspace{4cm}=\e\left[\int_{0}^{b} \mathcal{H}^{(q)}(b,x) \int_{(-x-Y,-x)}H_f(Y,x,y)\Pi(\ud y)\ud x\right],
\end{align}
where
\[
\begin{split}
H_f(a,x,y)&=\frac{\sigma^{2}}{2}\mathcal{O}^{(q)}(a,a+x+y)f(-a, -a)\\
&\hspace{1cm}+\int_{0}^a\mathcal{W}^{(\lambda)}(a,a+x+y,z)\int_{(-\infty,-z)}f(z-a, z+w-a)\Pi(\ud w)\ud z.
\end{split}
\]
$ii)$ For $q, b>0$, we have
\begin{align*}
\mathbf{n}\left(e^{-q\zeta}\mathbf{1}_{\left\{\tau_0^-<\tau_{-Y}^-<\zeta<\tau_b^+\right\}}\right)=\mathbb{E}\left[\int^{\infty}_{0} \mathcal{H}^{(q)}(b,x)\int_{(-x-Y, -x)}\left({e^{\Phi(q)(x+y+Y)}}-\frac{W^{(q)}(x+y+Y)}{W^{(q)}(a)}\right)\Pi(\ud y)\ud x\right].\notag\\
\end{align*}

\end{lemma}
\begin{proof}We first  prove part (i). Observe that the strong Markov property at time $\tau_0^-$ under $N$ implies
\begin{align}\label{idenlemma}
&\mathbf{n}\left(e^{-q \tau_{-Y}^-}f(X_{\tau_{-Y}^--}, X_{\tau_{-Y}^-},)\mathbf{1}_{\{\tau_0^-< \tau_{-Y}^-<\zeta<\tau_b^+\}}\right)=\mathbf{n}\left(e^{-q \tau_0^-}{\mathcal{L} ( X_{\tau_0^-})}\mathbf{1}_{E_2}\right),
\end{align}
where $E_2$ is defined in (\ref{conj2}) and for $z<0$, 
\[
{\mathcal{L}(z)}=\e_{z}\left[e^{-q\tau_{-Y}^-}f( X_{\tau_{-Y}^--}, X_{\tau_{-Y}^-} )\mathbf{1}_{\left\{\tau_{-Y}^-<\tau_0^+\right\}}\mathbf{1}_{\left\{ z>-Y\right\}}\right] .
\]
Next, we use the  same arguments used in  the  proof of Lemma 3, i.e.  we apply the Master's formula,  under $\mathbf{n}$, and Lemma 1 to the right-hand side of (\ref{idenlemma}) to deduce
\begin{align}
&\mathbf{n}\left(e^{-q \tau_0^-}\mathcal{L} ( X_{\tau_0^-})\mathbf{1}_{E_2}\right)\notag\\
&\hspace{1cm}=\int_{0}^{b} \mathcal{H}^{(q)}(b,x)\int_{(-\infty,-x)}\e_{x+y}\left[\mathbf{1}_{\left\{x+y>-Y\right\}}e^{-q\tau_{-Y}^-}f( X_{\tau_{-Y}^--}, X_{\tau_{-Y}^-} )\mathbf{1}_{\left\{\tau_{-Y}^-<\tau_0^+\right\}}\right]\Pi(\ud y)\ud x.\notag
\end{align}
The result   follows from  the following identity which is consequence of Theorem 1, part (ii),
\begin{equation*}
\begin{split}
\e_{x+y}\left[e^{-\lambda\tau_{-a}^-}g(X_{\tau_{-a}^--}, X_{\tau_{-a}^-}) \mathbf{1}_{\left\{\tau_{-a}^-<\tau_0^+\right\}}\right] &=\frac{\sigma^{2}}{2}\mathcal{O}^{(\lambda)}(a,a+x+y)g(-a,-a)\\
&\hspace{-1cm}+\int_{0}^a\mathcal{W}^{(\lambda)}(a,a+x+y,z)\int_{(-\infty,-z)}g(z-a, z+w-a)\Pi(\ud w)\ud z.
\end{split}
\end{equation*}
Part (ii) follows from similar reasoning as above. More precisely,  we  apply the strong Markov property at time $\tau_0^-$ under $\mathbf{n}$. Thus, 
\begin{align*}
\mathbf{n}\left(e^{-q\zeta}\mathbf{1}_{\left\{\tau_0^-<\tau_{-Y}^-<\zeta<\tau_b^+\right\}}\right)=\mathbf{n}\left(e^{-q \tau_0^-}{\mathcal{L}_1 (X_{\tau_0^-})}\mathbf{1}_{E_2}\right),
\end{align*}
where  for $z<0$, 
\[
 \mathcal{L}_1(z)=\e_{z}\left[e^{-q\tau_{0}^+}\mathbf{1}_{\left\{\tau_{-Y}^-<\tau_0^+\right\}}\mathbf{1}_{\left\{ z>-Y\right\}}\right] .
\]
Again, using similar arguments as those used in  the  proof of Lemma 3, we get
\begin{align}
&\mathbf{n}\left(e^{-q \tau_0^-} \mathcal{L}_1(X_{\tau_0^-})\mathbf{1}_{E_2}\right)=\int_{0}^{b} \mathcal{H}^{(q)}(b,x)\int_{(-\infty,-x)}\e_{x+y}\left[e^{-q\tau_{0}^+}\mathbf{1}_{\left\{\tau_{-Y}^-<\tau_0^+\right\}}\mathbf{1}_{\left\{ x+y>-Y\right\}}\right]\Pi(\ud y)\ud x.\notag
\end{align}
The proof now follows from  (\ref{fi2}).
\end{proof}

In order to describe the event in (v), we need the following Lemma which is proved  with similar arguments as those used previously. In particular, we get functions $\mathcal{C}_f(q)$ and $\mathcal{F}(q)$.
\begin{lemma} For $q,b,\sigma>0$, it follows
\begin{align}
\hspace{-4cm}\mathbf{n}&\left(e^{-q\tau^{-}_{-Y}}f(X_{\tau^{-}_{-Y}-}, X_{\tau^{-}_{-Y}})\mathbf{1}_{\{\tau_0^-=0,\tau_{-Y}^-<\zeta\}}\right)=-\frac{\sigma^4}{4}\e\left[f(-Y,-Y)\frac{\partial}{\partial x}\mathcal{O}^{(q)}(Y,Y)\right]\notag\\
&\hspace{3cm}+\frac{\sigma^2}{2}\e\left[\int_0^Y\int_{(-\infty,-y)}f(y-Y,y+z-Y)\mathcal{O}^{(q)}(Y,Y-y)\Pi(\ud z)\ud y\right].\label{eqlem51}\\
\mathbf{n}\Big(e^{-q\zeta}&\mathbf{1}_{\{\tau_0^-=0,\tau_{-Y}^-<\zeta\}}\Big)=-\frac{\sigma^4}{4}\e\left[e^{-\Phi(q)Y}\frac{\partial}{\partial x}\mathcal{O}^{(q)}(Y,Y)\right]\notag\\
&\hspace{4cm}+\frac{\sigma^2}{2}\e\left[\int_0^Y\int_{(-\infty,-y)}e^{\Phi(q)(y+z-Y)}\mathcal{O}^{(q)}(Y,Y-y)\Pi(\ud z)\ud y\right].\label{eqlem52}
\end{align}
\end{lemma}
\begin{proof}  Let $\bar{\mathbf{n}}$ denote the excursion measure of the process reflected at its supremum. Then 
performing similar computations as those that appear in Theorem 3.10 in \cite{KKR} (see bottom of p. 138), we observe
\begin{align*}
\bar{\mathbf{n}}\Big(e^{-q\tau_{Y}^+}f(-X_{\tau_{Y}^+-},-X_{\tau_{Y}^+})\mathbf{1}_{\{X_{\tau_{Y}^+}=Y,\tau_{Y}^+<\zeta\}}\Big)=-
\frac{\sigma^2}{2}\mathbb{E}\left[f(-Y,-Y)\frac{\partial}{\partial x}\mathcal{O}^{(q)}(Y,Y)\right].
\end{align*}
On the other hand, we also can obtain the following identity (see top of p. 139 in \cite{KKR})
\[
\begin{split}
\bar{\mathbf{n}}\Big(e^{-q\tau_{Y}^+}f(-X_{\tau_{Y}^+-},&-X_{\tau_{Y}^+})\mathbf{1}_{\{X_{\tau_{Y}^+}>Y,\tau_{Y}^+<\zeta\}}\Big)\\
&=\mathbb{E}\left[\int_0^Y\int_{(-\infty,-y)}f(y-Y,y+z-Y)\mathcal{O}^{(q)}(Y,Y-y)\Pi(\ud z)\ud y\right]\notag.
\end{split}
\]
The identity in (\ref{eqlem51}) now follows from Theorem 3,  part (i), in  \cite{cem} which  implies the following relation between the excursion measure $\mathbf{n}$ and the excursion measure of the process reflected at its supremum $\bar{\mathbf{n}}$. For any $a>0$, we have
\begin{align}
\mathbf{n}\Big(e^{-q\tau_{-a}^-}f(X_{\tau_{-a}^--}&,X_{\tau_{-a}^-})\mathbf{1}_{\{  \tau_0^-=0, \ \tau_{-a}^-<\zeta\}}\Big)=\frac{\sigma}{2} \bar{\mathbf{n}}\Big(e^{-q\tau_{a}^+}f(-X_{\tau_{a}^+-},-X_{\tau_{a}^+})\mathbf{1}_{\{\tau_{a}^+<\zeta\}}\Big).\notag
\end{align}
In order to prove (\ref{eqlem52}), we use again the relation  between the excursion measures $\mathbf{n}$ and  $\bar{\mathbf{n}}$. Hence recalling that $Y$ is an independent r.v. of $X$,   using  the strong Markov property at time $\tau^-_{-Y}$ and identity (\ref{fi2}), we observe
\begin{align}
\mathbf{n}\Big(e^{-q\zeta}\mathbf{1}_{\{ \tau_0^-=0,\ \tau_{-Y}^-<\zeta\}}\Big)&=\mathbf{n}\Big(e^{-q\tau_{-Y}^-}\e_{-X_{\tau_{-Y}^-}}\Big[e^{-q\tau_{0}^+}\mathbf{1}_{\{\tau_{0}^+<\infty\}}\Big]\mathbf{1}_{\{ \tau_0^-=0,\ \tau_{-Y}^-<\zeta\}}\Big)\notag\\
&=\mathbf{n}\Big(e^{-q\tau_{-Y}^-}e^{\Phi(q)X_{\tau_{-Y}^-}}\mathbf{1}_{\{\tau_0^-=0,\ \tau_{-Y}^-<\zeta\}}\Big)\notag\\
&=\frac{\sigma^2}{2}\bar{\mathbf{n}}\Big(e^{-q\tau_{Y}^+}e^{-\Phi(q)X_{\tau_{Y}^+}}\mathbf{1}_{\{\tau_{Y}^+<\zeta\}}\Big)\notag.
\end{align}
The proof is completed once we follow the same steps as in the proof of (\ref{eqlem51}).

\end{proof}

\section{The case $x>0$.}
We finish the paper with an explanation of how to get the  Gerber-Shiu penalty function at time of bankruptcy $T_B$ for any  initial surplus $x\geq0$. We first consider the two possible  behaviours of   the first excursion away from $0$ under the restriction that it does not goes above level $b$: either bankruptcy occurs in the first excursion away from $0$, or  there is no bankruptcy and there is a new excursion starting from  $0$.  In the first scenario, we have three possibilities: (a) the process $X$ jumps at $\tau_0^-$ below the level   $-Y$, (b)  the process $X$ jumps at $\tau_0^-$ in the interval   $(-Y, 0)$ and then goes below the level $-Y$ continuously or by a jump or (c) the process $X$ hits $0$ continuously at $\tau_0^-$ after an independent exponential random time $\mathbf{e}_\lambda$.  In the second scenario, we have two possibilities: (d) either the process $X$ hits $0$ continuously at $\tau_0^-$ before the independent exponential random time $\mathbf{e}_\lambda$ and then a new excursion starts or (e) the process $X$ jumps at $\tau_0^-$ in the interval   $(-Y, 0)$ and  it returns to 0 without going below the level $-Y$.   Therefore using  the strong Markov property and the independence between the excursions, we obtain
\begin{align*}\label{gsx}
\phi_f(x,q,b)&=\e_x\left[e^{-q\tau_0^-}f(X_{\tau_0^- -}, X_{\tau_0^-})\mathbf{1}_{\left\{\tau_0^-<\tau_b^+, X_{\tau_0^-}<-Y\right\}}\right]+f(0,0)\e_x\left[e^{-q\tau_0^-}\mathbf{1}_{\left\{ X_{\tau_0^--}=0,\mathbf{e}_{\lambda}<\tau_0^-<\tau_b^+\right\}}\right]\notag\\
&\hspace{1cm}+\e_x\left[e^{-q\tau_0^-}\mathbf{1}_{\{\tau_0^-<\tau_b^+,  0>X_{\tau_0^-}>-Y\}}\e_{X_{\tau_0^-}}\left[e^{-q\tau_{-Y}^-}f(X_{\tau_{-Y}^-},X_{\tau_{-Y}^--})\mathbf{1}_{\left\{\tau_{-Y}^-<\tau_{0}^+\right\}}\right]\right]\notag\\
&\hspace{1cm}+\e_x\left[e^{-q\tau_0^-}\mathbf{1}_{\left\{ X_{\tau_0^--}=0,\tau_0^-<\tau_b^+\land \mathbf{e}_{\lambda}\right\}}\right]\phi_f(0,q,b)\\
&\hspace{1cm}+\e_x\left[e^{-q\tau_0^-}\mathbf{1}_{\left\{\tau_0^-<\tau_b^+,  0>X_{\tau_0^-}>-Y\right\}}\e_{X_{\tau_0^-}}\left[e^{-q\tau_{0}^+}\mathbf{1}_{\{\tau_{-Y}^->\tau_{0}^+\}}\right]\right]\phi_f(0,q,b).\notag\\
\end{align*}
Observe that the first term, which describes the case in (a),  follows from  identity (\ref{fi1}), i.e.
\begin{align*}
\e_x\left[e^{-q\tau_0^-}f(X_{\tau_0^- -}, X_{\tau_0^-})\mathbf{1}_{\left\{\tau_0^-<\tau_b^+, X_{\tau_0^-}<-Y\right\}}\right]=\mathbb{E}\left[\int_0^{b}\int_{(-\infty,-(y+Y))}f(y,y+u)\mathcal{W}^{(q)}(b, x, y)\Pi(\ud u )\ud y\right].
\end{align*}
The third term, which is described in (b), is precisely the same term that is computed in   (\ref{shitty3}), then
\begin{align*}
\e_x\left[e^{-q\tau_0^-}\mathbf{1}_{\{\tau_0^-<\tau_b^+, 0>X_{\tau_0^-}>-Y\}}\e_{X_{\tau_0^-}}\left[e^{-q\tau_{-Y}^-}f(X_{\tau_{-Y}^-},X_{\tau_{-Y}^--})\mathbf{1}_{\left\{\tau_{-Y}^-<\tau_{0}^+\right\}}\right]\right]={\mathcal{I}_f(x,q,b)},
\end{align*}
where 
\[
\begin{split}
&\mathcal{I}_f(x,q,b)\\
&\hspace{.6cm}=\e\bigg[\int_0^{b}\int_{-(y+Y)}^{-y}\int_0^Y\int_{-\infty}^{-v}f(v+w-Y,v-Y)\mathcal{W}^{(q)}(Y,y+u+Y,v)
\mathcal{W}^{(q)}(b,x,y)\Pi(\ud w)dv\Pi(\ud u)dy\bigg].\notag
\end{split}
\]
For the second and fourth terms, which are described in (c) and (d) respectively,   we use identity (\ref{fi1}) to deduce
\begin{align*}
\e_x\left[e^{-q\tau_0^-}\mathbf{1}_{\left\{ X_{\tau_0^--}=0,\mathbf{e}_{\lambda}<\tau_0^-<\tau_b^+\right\}}\right]&=\e_x\left[e^{-q\tau_0^-}(1-e^{-\lambda\tau_0^-})\mathbf{1}_{\left\{ X_{\tau_0^--}=0,\tau_0^-<\tau_b^+\right\}}\right]\\
&=\frac{\sigma^2}{2}\Big(\mathcal{O}^{(q)}(b,x)-\mathcal{O}^{(q+\lambda)}(b,x)\Big),
\end{align*}
and 
\begin{align*}
\e_x\left[e^{-q\tau_0^-}\mathbf{1}_{\left\{ X_{\tau_0^--}=0,\tau_0^-<\tau_b^+\land \mathbf{e}_{\lambda}\right\}}\right]=\frac{\sigma^2}{2}\mathcal{O}^{(q+\lambda)}(b,x).
\end{align*}
Finally for the fifth term, that corresponds to the case  (e), we  use Theorem \ref{fi} and get
\begin{align*}
\e_x&\left[e^{-q\tau_0^-}\mathbf{1}_{\{\tau_0^-<\tau_b^+, 0>X_{\tau_0^-}>-Y\}}\e_{X_{\tau_0^-}}\left[e^{-q\tau_{0}^+}\mathbf{1}_{\{\tau_{-Y}^->\tau_{0}^+\}}\right]\right]=\e_x\left(e^{-q\tau_0^-}1_{\{\tau_0^-<\tau_b^+, X_{\tau_0^-}>-Y\}}\frac{W^{(q)}(X_{\tau_0^-}+Y)}{W^{(q)}(Y)}\right)\notag\\
&\hspace{6.5cm}=\e\left[\int_0^{b}\int_{(-y-Y, -y)}\frac{W^{(q)}(y+u+Y)}{W^{(q)}(Y)}\mathcal{W}^{(q)}(b,x,y)\Pi(\ud u)\ud y\right].
\end{align*}
Therefore putting all these identities together give us 
\begin{align*}
\phi_f(x,q,b)&=\mathbb{E}\left[\int_0^{b}\int_{(-\infty,-(y+Y))}f(y,y+u)\mathcal{W}^{(q)}(b, x, y)\Pi(\ud u )\ud y\right]\\
&\hspace{1cm}+f(0,0)\frac{\sigma^2}{2}\Big(\mathcal{O}^{(q)}(b,x)-\mathcal{O}^{(q+\lambda)}(b,x)\Big)+{\mathcal{I}_f(x,q,b)}+ \phi_f(0,q,b)\frac{\sigma^2}{2}\mathcal{O}^{(q+\lambda)}(b,x)\\
&\hspace{1cm}+\e\left[\int_0^{b}\int_{(-y-Y, -y)}\frac{W^{(q)}(y+u+Y)}{W^{(q)}(Y)}\mathcal{W}^{(q)}(b,x,y)\Pi(\ud u)\ud y\right]\phi_f(0,q,b).\notag\\
\end{align*}
Similar results can be obtained for the case when $b$ goes to $\infty$, we leave the details to the interested reader.

\end{document}